\documentclass[11pt,reqno]{amsart}
 \usepackage{tikz}
   \usetikzlibrary{positioning}
  \usepackage{tkz-euclide}
\usepackage{rotating}
\usepackage{color,xcolor}
\usepackage{leftidx}
\usepackage{amsmath,amssymb,amscd,graphicx, amsthm,
  enumerate, hyperref} 
\calclayout

\usepackage{mathrsfs}
\usepackage{mathtools}
\usepackage[all]{xy}
\usepackage{amsmath}
\usepackage{color,xcolor}
\definecolor{darkred}{rgb}{1,0,0} 
\definecolor{darkgreen}{rgb}{0,0.8,0}
\definecolor{darkblue}{rgb}{0,0,1}

\hypersetup{colorlinks,
linkcolor=darkblue,
filecolor=darkgreen,
urlcolor=darkred,
citecolor=darkgreen}
\newcommand{\norm}[1]{\left\lVert#1\right\rVert}

\numberwithin{equation}{section}

\theoremstyle{plain}

\theoremstyle{plain}
\newtheorem{theorem}{Theorem}
\numberwithin{theorem}{section}
\newtheorem{proposition}[theorem]{Proposition}
\newtheorem{lemma}[theorem]{Lemma}

\theoremstyle{definition}
\newtheorem{examplex}[theorem]{Example}
\newenvironment{example}
  {\pushQED{\qed}\examplex}
  {\popQED\endexamplex}
\usepackage{accents}
\theoremstyle{definition}
\newtheorem{remark}[theorem]{Remark}

\newtheorem{notation}[theorem]{Notation}

\newtheorem*{theorem*}{Corollary of Conjecture 4.1}

\newcommand{\interior}[1]{%
  {\kern0pt#1}^{\mathrm{o}}%
}

\DeclareMathOperator{\Imp}{Im}
\DeclareMathOperator{\Realp}{Re}

\DeclareMathOperator{\dist}{dist}

\newcommand{\C}{\mathbb{C}}

\title{Supnorm estimates for $\bar\partial$ on product domains in $\C^n$ }
\author{Martino Fassina}
\address{Department of Mathematics, University of Illinois at Urbana-Champaign, 1409 West Green
Street, Urbana, IL 61801, USA}
\email{fassina2@illinois.edu}

\author{Yifei Pan}
\address{Department of Mathematical Sciences, Purdue University Fort Wayne, 2101 East Coliseum Boulevard,
Fort Wayne, IN 46805, USA}
\email{pan@pfw.edu}

\begin{document}

\begin{abstract}
Let $\Omega\subset\C^n$ be a product of one-dimensional open bounded domains with $C^{1,\alpha}$ boundary, where $0<\alpha<1$. Using methods from complex analysis in one variable, we construct an integral operator that solves $\bar\partial$ in $\Omega$ with supnorm estimates when the datum is in $C^{n-1,\alpha}(\Omega)$. 

\end{abstract}
\maketitle
\section{Introduction and Main Result}
Let $\Omega$ be a domain in $\C^n$ and $f$ a $\bar\partial$-closed $(0,1)$ form on $\Omega$. A fundamental question in complex analysis is to establish the existence of solutions to the inhomogeneous Cauchy-Riemann equations $\bar\partial u=f$ that satisfy supnorm estimates in $\Omega$. 
The investigation of $L^{\infty}$ estimates for $\bar\partial$ was pioneered by Henkin \cite{H70} and Grauert-Lieb \cite{GL71} in the early 1970s. They independently proved the existence of solutions to the equation $\bar\partial u=f$ satisfying supnorm estimates when $\Omega$ is a strictly pseudoconvex domain with smooth boundary and $f$ is smooth in $\Omega$. In 1971 Henkin constructed an integral operator that solves $\bar\partial$ with supnorm estimates on the bidisc $D^2$ whenever the datum $f$ is in $C^1(\overline{D^2})$. 
A full proof of $L^{\infty}$ estimates for Henkin's solution appears in \cite{FLZ11}. Solutions with supnorm estimates have also been established on strictly pseudoconvex domains with piecewise smooth boundary \cite{RS73} and on pseudoconvex analytic polyhedra \cite{SH80}, once again under the assumption that the initial datum $f$ is smooth in $\Omega$. 

In this paper we consider domains $\Omega$ that are product of one-dimensional bounded domains. We give an explicit formula for a solution ${\bf T}f$ of the equation $\bar\partial u=f$ that satisfies supnorm estimates in $\Omega$. Below is the precise statement of our result. Note that, given a positive integer $m$, and $0<\alpha<1$, we write $C^{m,\alpha}$ for the corresponding H\"older space. 
\begin{theorem}\label{MMMAIN}
Let $D_1,\dots,D_n\subset \C$ be open bounded domains with $C^{1,\alpha}$ boundary, where $0<\alpha<1$. Consider the product domain $\Omega=D_1\times\dots\times D_n$. Let $f=f_1d\bar{z}_1+\dots +f_nd\bar{z}_n$ be a $\bar\partial$-closed $(0,1)$ form on $\Omega$ with components $f_j\in C^{n-1,\alpha}(\Omega)$. There exists ${\bf T}f\in C^{1,\alpha}(\Omega)$ solving
\begin{equation*}
\bar\partial {\bf T}f=f  \text{ in }\Omega
\end{equation*}
and satisfying the supnorm estimate 
\begin{equation*}
\norm{{\bf T}f}_{L^{\infty}(\Omega)}\leq C\norm{f}_{L^{\infty}(\Omega)}
\end{equation*}
for a constant $C$ independent of $f$.
\end{theorem} 


Our formula for the operator ${\bf T}$ is inspired by the recent work of Chen and McNeal on product domains in $\C^2$ \cite{CM}. The novelty of our approach is the use of purely one-variable methods. \footnote{After the first version of this paper appeared on the arXiv, the preprint \cite{CM19} was posted by Chen and McNeal. In that paper, they also obtain the formula for the solution operator {\bf T} on a product domain $\Omega$ in $\C^n$, and they essentially prove norm estimates from the space $W^{n-1,p}(\Omega)$ to $L^p(\Omega)$, for $p\in [1,\infty]$. }

The paper is organized as follows. Section \ref{S1} contains the necessary background material from the one-variable theory. In Sections \ref{sectionC2} and \ref{sectionC3} we present detailed proofs of Theorem \ref{MMMAIN} in the special cases of dimension two and three respectively. The two-dimensional case gives a simple proof of Henkin's results \cite{H71}, while the three-dimensional situation provides the insight to deal with the general case. The full proof of Theorem \ref{MMMAIN} is carried out in Section \ref{sectionCn}.  In Section \ref{S6} we define a new operator, denoted by ${\bf \widetilde{T}}$, which extends ${\bf T}$ to a wider class of $(0,1)$ forms. We discuss the possible role of ${\bf \widetilde{T}}$ in the context of weak solutions to $\bar\partial$ with supnorm estimates on product domains.
\section{Preliminaries}\label{S1}
Let $D\subset\C$ be an open bounded domain. For an integrable function $f$ on $D$ define  
$$Tf(z):=-\frac{1}{2\pi i}\int_D\frac{f(\zeta)}{\zeta-z}\, d\bar\zeta\wedge d\zeta.$$
We summarize here some well-known facts about this integral operator. For additional information we refer to the classic paper \cite{NW} for the case of the complex disc and to \cite{V} for general domains.
\begin{proposition}\cite[Theorem 1.32]{V}\label{old}
The following hold:
\begin{itemize}
\item $\partial_{\bar{z}} Tf=f$ in weak sense for any $f\in L^1(D)$.
\item Assume that $D$ has $C^{m+1,\alpha}$ boundary and $f\in C^{m,\alpha}(D)$ for some $0<\alpha<1$ and $m\geq 0$. Then $Tf\in C^{m+1,\alpha}(D)$ and $T\colon C^{m,\alpha}(D)\to C^{m+1,\alpha}(D)$ is a continuous linear operator. Moreover, there exists a linear bounded operator $\Pi\colon C^{m,\alpha}(D)\to C^{m,\alpha}(D)$, defined by a singular integral, such that $\Pi f=\partial_z Tf$ for every $f\in C^{m,\alpha}(D)$.
\end{itemize}
\end{proposition}

We exploit the one-dimensional theory in a very natural way to study solutions to the operator $\bar\partial$ on product domains $\Omega$ in $\C^n$. Let $\Omega=D_1\times\dots\times D_n$ for open bounded domains $D_j\subset\C$. For each $k=1,\dots,n$ we define the ``slice" operator $T^k$ on integrable functions $f\in L^1(\Omega)$ by
\begin{equation}\label{sliceop}
T^kf(z):=-\frac{1}{2\pi i}\int_{D_k}\frac{f(z_1,\dots,\zeta_k,\dots, z_n)}{\zeta_k-z_k}\, d\bar\zeta_k\wedge d\zeta_k.
\end{equation}
Using the operators $T^k$ we will give an integral formula for a solution of the $\bar\partial$-equation in $\Omega$ that satisfies a supnorm estimate. The key observation is that, by Proposition \ref{old}, 
\begin{equation}\label{funda}
f\in C^{1,\alpha}(\Omega) \text{ implies }\partial_{\bar{z}_k}T^kf=f \text{ in }\Omega. 
\end{equation}
 
 In our arguments, we will use repeatedly the two following elementary lemmas, of which, for the sake of completeness, we provide the proofs.

\begin{lemma}\label{lemma1}
Let $D\subset\C$ be a an open bounded domain and let $\alpha<2$. There exists a constant $C$ such that 
\begin{equation*}
\int_D \frac{|d\bar{\zeta}\wedge d\zeta|}{|\zeta-z|^{\alpha}}\leq C
\end{equation*}
holds for every $z\in D$.
\end{lemma}
\begin{proof}
Let $B_R(0)$ be a ball of radius $R>0$ centered at the origin such that $D\subset B_R(0)$. Fix $z\in D$. Then 
\begin{equation*}
\begin{split}
\int_D \frac{|d\bar{\zeta}\wedge d\zeta|}{|\zeta-z|^{\alpha}} \leq \int_{B_{2R}(0)} \frac{|d\bar{\zeta}\wedge d\zeta|}{|\zeta-z|^{\alpha}}&=\int_{B_R(z)} \frac{|d\bar{\zeta}\wedge d\zeta|}{|\zeta-z|^{\alpha}}+\int_{B_{2R}(0)\setminus B_R(z)} \frac{|d\bar{\zeta}\wedge d\zeta|}{|\zeta-z|^{\alpha}}\\ &\leq \int_{0}^{2\pi}\int_{0}^{R}\frac{2r drd\theta}{r^{\alpha}}+\int_{B_{2R}(0)}\frac{|d\bar{\zeta}\wedge d\zeta|}{R^{\alpha}}=\frac{4\pi R^{2-\alpha}}{2-\alpha}+2\pi R^{2-\alpha}.
\end{split}
\end{equation*}
Note that in the last equality we have used the hypothesis $\alpha<2$.
\end{proof}
\begin{lemma}\label{lemma2}
Let $D\subset\C$ be an open bounded domain with $C^1$ boundary $\partial D$ and let $\alpha<1$. There exists a constant $C$ such that 
\begin{equation*}
\int_{\partial D}\frac{|d\zeta|}{|\zeta-z|^{\alpha}}\leq C
\end{equation*}
holds for every $z\in D$.
\end{lemma}
\begin{proof}
First cover $\partial D$ with balls $B_j$ of the same radius $\delta>0$ whose centers belong to $\partial D$. Let $\delta$ be sufficiently small so that in each ball $B_j$ the boundary $\partial D$ is the graph of a $C^1$ function. Since $\partial D$ is compact, we can extract a finite subcover of $\partial D$, say $B_1,\dots,B_m$. Note that there exists $\epsilon>0$ such that 
\begin{equation}\label{dist}
\dist\big(\partial D, D\setminus \cup_{j=1}^m B_j\big)>\epsilon.
\end{equation} 
In fact, assume by contradiction that there exists a sequence of boundary points $\zeta_k$ and a sequence $z_k$ of points of $D$ such that $|\zeta_k-z_k|\to 0$ as $k\to\infty$. Since $\partial D$ is compact, there exists a subsequence $\zeta_{k_n}$ converging to a boundary point $\zeta_0$. By the triangle inequality, we have $|z_{k_n}-\zeta_0|\leq |z_{k_n}-\zeta_{{k}_n}|+|\zeta_{{k}_n}-\zeta_0|$. Hence $z_{k_n}\to \zeta_0$ as $n\to\infty$. This is absurd, since there exists a neighborhood of $\zeta_0$ entirely contained in $\cup_{j=1}^mB_j$. Now that we have proved \eqref{dist}, we consider two separate cases.

CASE 1: If $z\in D\setminus \cup_{j=1}^m B_j$, then $|\zeta-z|\geq \epsilon$ for all $\zeta\in\partial D$. Hence
\[\int_{\partial D}\frac{|d\zeta|}{|\zeta-z|^{\alpha}}\leq \frac{1}{\epsilon^{\alpha}}|\partial D|,\]
where $|\partial D|$ denotes the measure of $\partial D$. 

CASE 2: If $z\in\cup_{j=1}^m B_j$, then without loss of generality assume $z \in B_1$. We have
\begin{equation*}
\int_{\partial D}\frac{|d\zeta|}{|\zeta-z|^{\alpha}}=\sum_{j=1}^m\int_{\partial D\cap B_j}\frac{|d\zeta|}{|\zeta-z|^{\alpha}}.
\end{equation*}
Note, for $\zeta\in\partial D\cap B_j$ with $j=2,\dots,m$, that $|z-\zeta|>\epsilon$, and thus
\begin{equation*}
\int_{\partial D\cap B_j}\frac{|d\zeta|}{|\zeta-z|^{\alpha}}\leq \frac{1}{\epsilon^{\alpha}}|\partial D|,\quad j=2,\dots,m.
\end{equation*}
It is therefore enough to estimate the integral 
\begin{equation*}
\int_{\partial D\cap B_1}\frac{|d\zeta|}{|\zeta-z|^{\alpha}}.
\end{equation*}
We choose a coordinate system such that the ball $B_1$ is centered at the origin and $\partial D\cap B_1$ is a subset of the real axis. More specifically, we can achieve $\partial D\cap B_1=\{z=x+iy\in\C\,\vert\, y=0, -\delta\leq x \leq \delta\}$. Hence 
\begin{equation*}
\int_{\partial D\cap B_1}\frac{|d\zeta|}{|\zeta-z|^{\alpha}}=\int_{-\delta}^{\delta}\frac{dx}{|x-z|^{\alpha}}\leq \int_{-\delta}^{\delta}\frac{dx}{|x-\Realp z|^{\alpha}},
\end{equation*}
where the inequality follows from $|x-z|^2=(x-\Realp z)^2+(\Imp z)^2\geq |x-\Realp z|^2$. Letting $v=x-\Realp z$, 
\begin{equation*}
 \int_{-\delta}^{\delta}\frac{dx}{|x-\Realp z|^{\alpha}}=\int_{-\delta-\Realp z}^{\delta-\Realp z}\frac{dv}{|v|^{\alpha}}=\int_{-\delta-\Realp z}^{0}\frac{dv}{(-v)^{\alpha}}+\int_{0}^{\delta-\Realp z}\frac{dv}{v^{\alpha}}=\frac{(1+\Realp z)^{1-\alpha}}{1-\alpha}+\frac{(1-\Realp z)^{1-\alpha}}{1-\alpha}.
\end{equation*}
In the last equality we have used the hypothesis $\alpha<1$. Since here $|z|<\delta$, we obtain the estimate
\begin{equation*}
\int_{\partial D\cap B_1}\frac{|d\zeta|}{|\zeta-z|^{\alpha}}\leq \frac{2(1+\delta)^{1-\alpha}}{1-\alpha},
\end{equation*}
which concludes the proof.
\end{proof}

\section{Product domains in $\C^2$}\label{sectionC2}
Let $D_1,D_2\subset\C$ be open bounded domains with $C^{1,\alpha}$ boundary, where $0<\alpha<1$. Consider the product domain $\Omega=D_1\times D_2\subset\C^2$.
For a $(0,1)$ form $f=f_1d\bar{z}_1+f_2d\bar{z}_2$ on $\Omega$ with components $f_1,f_2\in C^{1,\alpha}(\Omega),$ we define
\begin{equation*}
{\bf T}f:=T^1f_1+T^2f_2-T^1T^2(\partial_{\bar{z}_1}f_2).
\end{equation*}
\begin{remark}
Note that $f_2\in C^{1,\alpha}(\Omega)$ implies $\partial_{\bar{z}_1}f_2\in C^{\alpha}(\Omega)$, and therefore ${\bf T}f$ is well defined.
\begin{lemma}\label{P} ${\bf T}f\in C^{1,\alpha}(\Omega)$.
\end{lemma}
\begin{proof}
Since $f_1,f_2\in C^{1,\alpha}(\Omega)$, it follows immediately that $T^1f_1, T^2f_2\in C^{1,\alpha}(\Omega)$. We just need to prove that $T^1T^2(\partial_{\bar{z}_1}f_2)\in C^{1,\alpha}(\Omega)$. Since $f_2\in C^{1,\alpha}(\Omega)$, then $\partial_{\bar{z}_1}f_2\in C^{\alpha}(\Omega)$. Hence Proposition \ref{old} implies \begin{equation}\label{11}
\begin{cases}
\partial_{\bar{z}_1}T^1T^2(\partial_{\bar{z}_1}f_2)=T^2(\partial_{\bar{z}_1}f_2)\in C^{\alpha}(\Omega)\\
\partial_{z_1}T^1T^2(\partial_{\bar{z}_1}f_2)=\Pi^1T^2(\partial_{\bar{z}_1}f_2)\in C^{\alpha}(\Omega).
\end{cases}
\end{equation}
Here $\Pi^1$ stands for the operator $\Pi$ applied to the variable $z_1$ (see Proposition \ref{old}). Now note that the operators $T^1$ and $T^2$ commute, by Fubini's theorem. We can thus argue in the same way for the variable $z_2$ and prove that \begin{equation}\label{22}
\begin{cases}
\partial_{\bar{z}_2}T^1T^2(\partial_{\bar{z}_1}f_2)\in C^{\alpha}(\Omega)\\
\partial_{z_2}T^1T^2(\partial_{\bar{z}_1}f_2)\in C^{\alpha}(\Omega).
\end{cases}
\end{equation}
From \eqref{11} and \eqref{22}, we conclude that $T^1T^2(\partial_{\bar{z}_1}f_2)\in C^{1,\alpha}$, as wanted.
\end{proof}
\end{remark}
\begin{lemma}\label{McNeal}
If the form $f$ is $\bar\partial$-closed, then $\bar\partial {\bf T}f=f$ in $\Omega$.
\end{lemma}
\begin{proof}
By \eqref{funda}, we have
\[\partial_{\bar{z}_1}{\bf T}f=\partial_{\bar{z}_1}T^1f_1+\partial_{\bar{z}_1}T^2f_2-\partial_{\bar{z}_1}T^1T^2(\partial_{\bar{z}_1}f_2)=f_1+T^2(\partial_{\bar{z}_1}f_2)-T^2(\partial_{\bar{z}_1}f_2)=f_1.\]
Note that $T^1T^2=T^2T^1$ by Fubini's theorem. Moreover, since $f$ is $\bar\partial$-closed, then $\partial_{\bar{z}_1}f_2=\partial_{\bar{z}_2}f_1$. Hence 
\[\partial_{\bar{z}_2}{\bf T}f=\partial_{\bar{z}_2}T^1f_1+\partial_{\bar{z}_2}T^2f_2-\partial_{\bar{z}_2}T^2T^1(\partial_{\bar{z}_2}f_1)=T^1(\partial_{\bar{z}_2}f_1)+f_2-T^1(\partial_{\bar{z}_2}f_1)=f_2.\]
We have thus proved that $\bar\partial {\bf T}f=f$ in $\Omega$. 
\end{proof}
\begin{remark}
Lemma \ref{McNeal} corresponds to \cite[Proposition 0.1]{CM}, where the same result is proved (using methods arising from several complex variables) under slightly weaker assumptions: the boundaries of the domains $D_1$ and $D_2$, as well as the functions $f_1, f_2$, are only required to be $C^1$ instead of $C^{1,\alpha}$.
\end{remark}
\begin{theorem}\label{theoremC2}
Let $D_1,D_2\subset\C$ be open bounded domains with $C^{1,\alpha}$ boundary, where $0<\alpha<1$. Consider the product domain $\Omega=D_1\times D_2$. Let $f=f_1d\bar{z}_1+f_2d\bar{z}_2$ be a $\bar\partial$-closed $(0,1)$ form on $\Omega$ with components $f_1,f_2\in C^{1,\alpha}(\Omega).$ Then ${\bf T}f\in C^{1,\alpha}(\Omega)$ and $\bar\partial {\bf T}f=f$ in $\Omega$. Moreover, ${\bf T}f$ satisfies the supnorm estimate
\begin{equation}\label{estimateC2}
\norm{{\bf T}f}_{L^{\infty}(\Omega)}\leq C\norm{f}_{L^{\infty}(\Omega)}
\end{equation}
for some constant $C$ independent of $f$.
\end{theorem}
\begin{proof}
We showed in Lemma \ref{P} and Lemma \ref{McNeal} that ${\bf T}f\in C^{1,\alpha}(\Omega)$ and $\bar\partial {\bf T}f=f$ in $\Omega$. We now prove the estimate \eqref{estimateC2}. First note that by the definition \eqref{sliceop} of the ``slice" operator $T^j$ and Lemma \ref{lemma1}, there exists a constant $C$ such that $\norm{T^jf_j}_{\infty}\leq C\norm{f_j}_{\infty}$ for $j=1,2$. Hence it is enough to estimate the function \begin{equation}\label{tbeC2}
T^1T^2(Df)=\frac{1}{(2\pi i)^2}\int_{D_1\times D_2}\frac{Df(\zeta_1,\zeta_2)}{(\zeta_1-z_1)(\zeta_2-z_2)}\,d\bar{\zeta}_1\wedge d\zeta_1\wedge d\bar{\zeta}_2\wedge d\zeta_2,
\end{equation}
where $Df:=\partial_{\bar{z}_1}f_2=\partial_{\bar{z}_2}f_1$. 
Recall that $|\zeta-z|^2=|\zeta_1-z_1|^2+|\zeta_2-z_2|^2$. We exploit the following identity:
\begin{equation}\label{decompositionC2}
\frac{1}{(\zeta_1-z_1)(\zeta_2-z_2)}=\frac{(\overline{\zeta_2-z_2})}{(\zeta_1-z_1)|\zeta-z|^2}+\frac{(\overline{\zeta_1-z_1})}{(\zeta_2-z_2)|\zeta-z|^2}.
\end{equation}
Substituting \eqref{decompositionC2} into \eqref{tbeC2} and recalling the two equivalent ways of expressing  $Df$, we write
\begin{equation}\label{here}
\begin{split}
(2\pi i)^2 T^1T^2(Df)&=\int_{D_1\times D_2}\partial_{\bar{\zeta}_2}f_1\frac{(\overline{\zeta_2-z_2})}{(\zeta_1-z_1)|\zeta-z|^2}\,d\bar{\zeta}_1\wedge d\zeta_1\wedge d\bar{\zeta}_2\wedge d\zeta_2\\&+\int_{D_1\times D_2}\partial_{\bar{\zeta}_1}f_2\frac{(\overline{\zeta_1-z_1})}{(\zeta_2-z_2)|\zeta-z|^2}\,d\bar{\zeta}_1\wedge d\zeta_1\wedge d\bar{\zeta}_2\wedge d\zeta_2\\&=I_1+I_2.
\end{split}
\end{equation}
We call $I_1$ and $I_2$ the two integrals appearing on the right side of \eqref{here}. Note that they can be estimated in the same way, by switching the the roles of $z_1$ and $z_2$. We can thus restrict our attention to $I_2$. The goal is now to show that there exists a constant $C$ such that 
\begin{equation}\label{goal}
\norm{I_2}_{L^{\infty}(\Omega)}\leq C\,\norm{f}_{L^{\infty}(\Omega)}.
\end{equation} 
For fixed $(z_1,z_2)\in D_1\times D_2$ we would like to apply Stokes' theorem in $I_2$ to remove the derivative from the function $f_2$. In order to do so, we need to avoid the point of singularity at $z_1$. We thus temporarily delete from $D_1$ a small ball $B_{\epsilon}(z_1)\subset\C$ of radius $\epsilon>0$ centered at $z_1$. We write 
\begin{equation*}
I_2=\int_{D_2}\Bigg{\{}\int_{D_1\setminus B_{\epsilon}(z_1)}\partial_{\bar{\zeta}_1}f_2\,\frac{(\overline{\zeta_1-z_1})}{(\zeta_2-z_2)|\zeta-z|^2} \,d\bar{\zeta}_1\wedge d\zeta_1 +\int_{ B_{\epsilon}(z_1)}\partial_{\bar{\zeta}_1}f_2\,\frac{(\overline{\zeta_1-z_1})}{(\zeta_2-z_2)|\zeta-z|^2} \,d\bar{\zeta}_1\wedge d\zeta_1\Bigg{\}}\, d\bar{\zeta}_2\wedge d\zeta_2.
\end{equation*}
Stokes' theorem then yields
\begin{equation}\label{imp}
\begin{split}
I_2&=\int_{D_2}\Bigg{\{}\int_{\partial D_1}f_2\,\frac{(\overline{\zeta_1-z_1})}{(\zeta_2-z_2)|\zeta-z|^2} \, d\zeta_1 - \int_{D_1\setminus B_{\epsilon}(z_1)}f_2\,\frac{\partial}{\partial{\bar{\zeta}_1}}\bigg{[}\frac{(\overline{\zeta_1-z_1})}{(\zeta_2-z_2)|\zeta-z|^2}\bigg{]} \,d\bar{\zeta}_1\wedge d\zeta_1\\&+\int_{ \partial B_{\epsilon}(z_1)} f_2\,\frac{(\overline{\zeta_1-z_1})}{(\zeta_2-z_2)|\zeta-z|^2} \,d\zeta_1 + \int_{ B_{\epsilon}(z_1)} \partial_{\bar{\zeta}_1}f_2\,\frac{(\overline{\zeta_1-z_1})}{(\zeta_2-z_2)|\zeta-z|^2} \,d\bar{\zeta}_1\wedge d\zeta_1\Bigg{\}} \,d\bar{\zeta}_2\wedge d\zeta_2.
\end{split}
\end{equation}
We now exploit the following simple observation: if $b_1,b_2$ are non-negative real numbers, $k_1, k_2$ are non-negative integers and $k=k_1+k_2$, then
\begin{equation}\label{st}
(b_1+b_2)^{k}\geq b_1^{k_1}b_2^{k_2}.
\end{equation}
We apply \eqref{st} with $b_1=|\zeta_1-z_1|$ and $b_2=|\zeta_2-z_2|$. Thus
\begin{equation}\label{tobeap}
|\zeta-z|^{2}=|\zeta_1-z_1|^2+|\zeta_2-z_2|^2\geq |\zeta_1-z_1|^{\frac{2k_1}{k}}|\zeta_2-z_2|^{\frac{2k_2}{k}}.
\end{equation}
Choosing $k_1=2, k_2=1$ in \eqref{tobeap}, we obtain
\begin{equation}\label{toep}
|\zeta-z|^2\geq|\zeta_1-z_1|^{\frac{4}{3}}|\zeta_2-z_2|^{\frac{2}{3}}.
\end{equation}
Hence
\begin{equation}\label{b}
\bigg{|}\frac{(\overline{\zeta_1-z_1})}{(\zeta_2-z_2)|\zeta-z|^2}\bigg{|}\leq \frac{1}{|\zeta_1-z_1|^{\frac{1}{3}}|\zeta_2-z_2|^{\frac{5}{3}}}.
\end{equation}
Applying Lemma \ref{lemma1} twice, we see that there exists a constant $C$ such that
\begin{equation}\label{a}
\int_{D_1\times D_2}\frac{|d\bar{\zeta}_1\wedge d\zeta_1\wedge d\bar{\zeta}_2\wedge d\zeta_2|}{|\zeta_1-z_1|^{\frac{1}{3}}|\zeta_2-z_2|^{\frac{5}{3}}}=\int_{D_1}\frac{|d\bar{\zeta}_1 \wedge d\zeta_1|}{|\zeta_1-z_1|^{\frac{1}{3}}}\int_{D_2}\frac{|d\bar{\zeta}_2\wedge d\zeta_2|}{|\zeta_2-z_2|^{\frac{5}{3}}}\leq C.
\end{equation}
Let $\chi_{\epsilon}(\zeta_1)$ be the characteristic function of the set $B_{\epsilon}(z_1)$. Hence
\begin{equation*}\label{qui}
\int_{ B_{\epsilon}(z_1)\times D_2} \partial_{\bar{\zeta}_1}f_2\,\frac{(\overline{\zeta_1-z_1})}{(\zeta_2-z_2)|\zeta-z|^2} \,d\bar{\zeta}_1\wedge d\zeta_1\wedge d\bar{\zeta}_2\wedge d\zeta_2=\int_{D_1\times D_2} \partial_{\bar{\zeta}_1}f_2\,\frac{\chi_{\epsilon}(\zeta_1)(\overline{\zeta_1-z_1})}{(\zeta_2-z_2)|\zeta-z|^2} \,d\bar{\zeta}_1\wedge d\zeta_1\wedge d\bar{\zeta}_2\wedge d\zeta_2.
\end{equation*}
Recall that $f_2\in C^{1,\alpha}(\Omega)$. By \eqref{b} and \eqref{a}, we can apply the dominated convergence theorem to conclude that 
\begin{equation}\label{7}
\lim_{\epsilon\to 0}\,\int_{ B_{\epsilon}(z_1)\times D_2} \partial_{\bar{\zeta}_1}f_2\,\frac{(\overline{\zeta_1-z_1})}{(\zeta_2-z_2)|\zeta-z|^2} \, d\zeta_1\wedge d\bar{\zeta}_2\wedge d\zeta_2=0.
\end{equation}
Note that \eqref{b} implies
\begin{equation*}
\bigg{|}\int_{ \partial B_{\epsilon}(z_1)\times D_2} \frac{(\overline{\zeta_1-z_1})}{(\zeta_2-z_2)|\zeta-z|^2} \, d\zeta_1\wedge d\bar{\zeta}_2\wedge d\zeta_2\bigg{|}\leq \int_{ \partial B_{\epsilon}(z_1)\times D_2}\frac{|d\zeta_1\wedge d\bar{\zeta}_2\wedge d\zeta_2|}{|\zeta_1-z_1|^{\frac{1}{3}}|\zeta_2-z_2|^{\frac{5}{3}}}=O(\epsilon^{\frac{5}{3}}).
\end{equation*}
Hence
\begin{equation}\label{8}
\lim_{\epsilon\to 0}\, \int_{ \partial B_{\epsilon}(z_1)\times D_2} f_2\,\frac{(\overline{\zeta_1-z_1})}{(\zeta_2-z_2)|\zeta-z|^2} \, d\zeta_1\wedge d\bar{\zeta}_2\wedge d\zeta_2=0.
\end{equation}
By \eqref{7} and \eqref{8}, letting $\epsilon\to 0$ in \eqref{imp} yields
\begin{equation}\label{quinta}
\begin{split}
I_2&=\int_{\partial D_1\times D_2}f_2\,\frac{(\overline{\zeta_1-z_1})}{(\zeta_2-z_2)|\zeta-z|^2} \, d\zeta_1\wedge d\bar{\zeta}_2\wedge d\zeta_2\\&- \int_{D_1\times D_2}f_2\,\frac{\partial}{\partial{\bar{\zeta}_1}}\bigg{[}\frac{(\overline{\zeta_1-z_1})}{(\zeta_2-z_2)|\zeta-z|^2}\bigg{]} \,d\bar{\zeta}_1\wedge d\zeta_1\wedge d\bar{\zeta}_2\wedge d\zeta_2.
\end{split}
\end{equation}
We now estimate the two integrals on the right side of \eqref{quinta}. First note, by \eqref{b}, Lemma \ref{lemma1} and Lemma \ref{lemma2}, that there exists a constant $C$ such that
\begin{equation}\label{seconda}
\bigg{|}\int_{\partial D_1\times D_2}\frac{(\overline{\zeta_1-z_1})}{(\zeta_2-z_2)|\zeta-z|^2} \, d\zeta_1\wedge d\bar{\zeta}_2\wedge d\zeta_2 \bigg{|}\leq C
\end{equation}
for every $(z_1,z_2)\in D_1\times D_2$.
Consider now the inequality \eqref{tobeap} with $k_1=1, k_2=2$. We get
\begin{equation*}
|\zeta-z|^2\geq|\zeta_1-z_1|^{\frac{2}{3}}|\zeta_2-z_2|^{\frac{4}{3}},
\end{equation*} 
which in turn yields the inequality
\begin{equation}\label{bbb}
\bigg{|}\frac{\partial}{\partial{\bar{\zeta}_1}}\bigg{[}\frac{(\overline{\zeta_1-z_1})}{(\zeta_2-z_2)|\zeta-z|^2}\bigg{]} \bigg{|}=\bigg{|}\frac{(\overline{\zeta_2-z_2})}{|\zeta-z|^4}\bigg{|}\leq\frac{1}{|\zeta_1-z_1|^{\frac{4}{3}}|\zeta_2-z_2|^{\frac{5}{3}}}.
\end{equation}
By \eqref{bbb} and Lemma \ref{lemma1}, there exists a constant $C$ such that
\begin{equation}\label{terza}
\bigg{|}\int_{D_1\times D_2}\frac{\partial}{\partial{\bar{\zeta}_1}}\bigg{[}\frac{(\overline{\zeta_1-z_1})}{(\zeta_2-z_2)|\zeta-z|^2}\bigg{]}\,d\bar{\zeta}_1\wedge d\zeta_1\wedge d\bar{\zeta}_2\wedge d\zeta_2\bigg{|}\leq C
\end{equation}
for every choice of $(z_1,z_2)\in D_1\times D_2$. Combining \eqref{quinta}, \eqref{seconda}, and \eqref{terza}, we conclude that
\begin{equation}\label{inh}
\norm{I_2}_{L^{\infty}(\Omega)}\leq C\,\norm{f}_{L^{\infty}(\Omega)}
\end{equation}
for some constant $C$. We have thus proved \eqref{goal}, as wanted.
\end{proof}

\section{Product domains in $\C^3$}\label{sectionC3}
Let $D_1,D_2,D_3\subset\C$ be open bounded domains with $C^{1,\alpha}$ boundary, where $0<\alpha<1$. Consider the product domain $\Omega=D_1\times D_2\times D_3$. For a $(0,1)$ form $f=f_1d\bar{z}_1+f_2d\bar{z}_2+f_3d\bar{z}_3$ on $\Omega$ with components $f_1,f_2,f_3\in C^{2,\alpha}(\Omega)$, we define
\begin{equation}\label{Tf}
{\bf T}f:=T^1f_1+T^2f_2+T^3f_3-T^1T^2(\partial_{\bar{z}_1}f_2)-T^1T^3(\partial_{\bar{z}_1}f_3)-T^2T^3(\partial_{\bar{z}_2}f_3)+T^1T^2T^3(\partial^2_{\bar{z}_1\bar{z}_2}f_3).
\end{equation}
Note that the regularity of the functions $f_j$ guarantees that ${\bf T}f$ is well defined.
\begin{lemma}\label{McNealC3}
${\bf T}f\in C^{1,\alpha}(\Omega)$. Moreover, if the form $f$ is $\bar\partial$-closed, then $\bar\partial {\bf T}f=f$ in $\Omega$.
\end{lemma}
\begin{proof}
The statements are proved in the same way as Lemma \ref{P} and Lemma \ref{McNeal}.
\end{proof}
\begin{theorem}\label{theoremC3}
Let $D_1,D_2,D_3\subset\C$ be open bounded domains with $C^{1,\alpha}$ boundary, where $0<\alpha<1$. Consider the product domain $\Omega=D_1\times D_2\times D_3$. Let $f=f_1d\bar{z}_1+f_2d\bar{z}_2+f_3d\bar{z}_3$ be a $\bar\partial$-closed $(0,1)$ form on $\Omega$ with components $f_1,f_2,f_3\in C^{2,\alpha}(\Omega).$ Then ${\bf T}f\in C^{1,\alpha}(\Omega)$ and $\bar\partial {\bf T}f=f$ in $\Omega$. Moreover, ${\bf T}f$ satisfies the supnorm estimate
\begin{equation}\label{estimateC3}
\norm{{\bf T}f}_{L^{\infty}(\Omega)}\leq C\norm{f}_{L^{\infty}(\Omega)}
\end{equation}
for some constant $C$ independent of $f$.
\end{theorem}
\begin{proof} 
We only need to prove the estimate \eqref{estimateC3}. As we have already observed in the proof of Theorem \ref{theoremC2}, Lemma \ref{lemma1} implies that there exists a constant $C$ such that $\norm{T^jf_j}_{L^{\infty}(\Omega)}\leq C\norm{f_j}_{L^{\infty}(\Omega)}$ for $j=1,2,3$. Additionally, the terms of the form 
$T^iT^j(\partial_{\bar{z}_i}f_j)$ can be estimated in the same way as in the 2-dimensional case considered in Theorem \ref{theoremC2}. Hence, letting $Df:=\partial^2_{\bar{z}_1\bar{z}_2}f_3=\partial^2_{\bar{z}_1\bar{z}_3}f_2=\partial^2_{\bar{z}_2\bar{z}_3}f_1$, it is enough to estimate the term 
\begin{equation}\label{T3}
T^1T^2T^3(Df)=-\frac{1}{(2\pi i)^3}\int_{D_1\times D_2\times D_3}\frac{Df(\zeta_1,\zeta_2,\zeta_3)}{(\zeta_1-z_1)(\zeta_2-z_2)(\zeta_3-z_3)}\, d\bar{\zeta}_1\wedge\dots\wedge d\zeta_3.
\end{equation}
Define $G:=|\zeta_2-z_2|^2|\zeta_3-z_3|^2+|\zeta_1-z_1|^2|\zeta_3-z_3|^2+|\zeta_1-z_1|^2|\zeta_2-z_2|^2$. We exploit the following higher-dimensional analog of identity \eqref{decompositionC2}:
\begin{equation}\label{decompositionC3}
\frac{1}{(\zeta_1-z_1)(\zeta_2-z_2)(\zeta_3-z_3)}=\frac{(\overline{\zeta_2-z_2})(\overline{\zeta_3-z_3})}{(\zeta_1-z_1)\,G}+\frac{(\overline{\zeta_1-z_1})(\overline{\zeta_3-z_3})}{(\zeta_2-z_2)\,G}+\frac{(\overline{\zeta_1-z_1})(\overline{\zeta_2-z_2})}{(\zeta_3-z_3)\,G}.
\end{equation}
Substituting \eqref{decompositionC3} into \eqref{T3} and recalling the three different ways of expressing $Df$, we obtain
\begin{equation} \label{threeintegrals}
\begin{split}
-(2\pi i)^3T^1T^2T^3(Df)&=\int_{D_1\times D_2\times D_3}\frac{\partial^2 f_1}{\partial\bar{\zeta}_2\partial\bar{\zeta}_3}\,\frac{(\overline{\zeta_2-z_2})(\overline{\zeta_3-z_3})}{(\zeta_1-z_1)\,G}\, d\bar{\zeta}_1\wedge\dots\wedge d\zeta_3\\&+ \int_{D_1\times D_2\times D_3}\frac{\partial^2 f_2}{\partial\bar{\zeta}_1\partial\bar{\zeta}_3}\,\frac{(\overline{\zeta_1-z_1})(\overline{\zeta_3-z_3})}{(\zeta_2-z_2)\,G}\, d\bar{\zeta}_1\wedge\dots\wedge d\zeta_3\\&+\int_{D_1\times D_2\times D_3}\frac{\partial^2 f_3}{\partial\bar{\zeta}_1\partial\bar{\zeta}_2}\,\frac{(\overline{\zeta_1-z_1})(\overline{\zeta_2-z_2})}{(\zeta_3-z_3)\,G}\, d\bar{\zeta}_1\wedge\dots\wedge d\zeta_3\\&=I_1+I_2+I_3.
\end{split}
\end{equation}
We call $I_1,I_2$ and $I_3$ the integrals appearing on the right side of \eqref{threeintegrals}. Note that they can be estimated in the same way by renaming the variables. Without loss of generality we thus restrict our attention to $I_3$. The goal is now to prove that there exists a constant $C$ such that 
\begin{equation}\label{goaaaal}
\norm{I_3}_{L^{\infty}(\Omega)}\leq C\norm{f}_{L^{\infty}(\Omega)}.
\end{equation}
Fix $(z_1,z_2,z_3)\in D_1\times D_2\times D_3$. As in the 2-dimensional case, in order to apply Stokes' theorem, we remove from $D_1$ a small ball $B_{\epsilon}(z_1)$ of radius $\epsilon>0$ centered at $z_1$. We obtain
\begin{equation}\label{again}
\begin{split}
I_3&=\int_{D_2\times D_3}\Bigg{\{}\int_{\partial D_1}\frac{\partial f_3}{\partial\bar{\zeta}_2}\,\frac{(\overline{\zeta_1-z_1})(\overline{\zeta_2-z_2})}{(\zeta_3-z_3)\,G}\,d\zeta_1-\int_{D_1\setminus B_{\epsilon}(z_1)}\frac{\partial f_3}{\partial\bar{\zeta}_2}\frac{\partial}{\partial\bar{\zeta}_1}\bigg{[}\frac{(\overline{\zeta_1-z_1})(\overline{\zeta_2-z_2})}{(\zeta_3-z_3)\,G}\bigg{]}d\bar{\zeta}_1\wedge d\zeta_1\\&+\int_{\partial B_{\epsilon}(z_1)}\frac{\partial f_3}{\partial\bar{\zeta}_2}\,\frac{(\overline{\zeta_1-z_1})(\overline{\zeta_2-z_2})}{(\zeta_3-z_3)\,G}\,d\zeta_1+\int_{B_{\epsilon}(z_1)}\frac{\partial^2 f_3}{\partial\bar{\zeta}_1\partial\bar{\zeta}_2}\,\frac{(\overline{\zeta_1-z_1})(\overline{\zeta_2-z_2})}{(\zeta_3-z_3)\,G}\,d\bar{\zeta}_1\wedge d\zeta_1\Bigg{\}} d\bar{\zeta}_2 \wedge\dots \wedge d\zeta_3.\\
\end{split}
\end{equation}
We now prove that the last two integrals in \eqref{again} disappear if we let $\epsilon\to 0$. First note that if $b_1,b_2, b_3$ are non-negative real numbers, $k_1, k_2, k_3$ are non-negative integers and $k=k_1+k_2+k_3$, then
\begin{equation}\label{st3}
(b_1+b_2+b_3)^{k}\geq b_1^{k_1}b_2^{k_2}b_3^{k_3}.
\end{equation}
We apply \eqref{st3} with $b_1=|\zeta_2-z_2|^2|\zeta_3-z_3|^2$, $b_2=|\zeta_1-z_1|^2|\zeta_3-z_3|^2$ and $b_3=|\zeta_1-z_1|^2|\zeta_2-z_2|^2$. Then
\begin{equation}\label{theG3}
G=b_1+b_2+b_3\geq b_1^{\frac{k_1}{k}}b_2^{\frac{k_2}{k}}b_3^{\frac{k_3}{k}}=|\zeta_1-z_1|^{2\frac{k_2+k_3}{k}}|\zeta_2-z_2|^{2\frac{k_1+k_3}{k}}|\zeta_3-z_3|^{2\frac{k_1+k_2}{k}}.
\end{equation}
Letting $k_1=k_2=1, k_3=6$ in \eqref{theG3}, we obtain
\begin{equation}\label{g}
\bigg{|}\frac{(\overline{\zeta_1-z_1})(\overline{\zeta_2-z_2})}{(\zeta_3-z_3)\,G}\bigg{|}\leq\frac{1}{|\zeta_1-z_1|^{\frac{3}{4}}|\zeta_2-z_2|^{\frac{3}{4}}|\zeta_3-z_3|^{\frac{3}{2}}}.
\end{equation}
By Lemma \ref{lemma1}, there exists a constant $C$ such that 
\begin{equation}\label{h}
\int_{D_1\times D_2\times D_3} \frac{|d\bar{\zeta}_1\wedge\dots\wedge d\zeta_3|}{|\zeta_1-z_1|^{\frac{3}{4}}|\zeta_2-z_2|^{\frac{3}{4}}|\zeta_3-z_3|^{\frac{3}{2}}} \leq C.
\end{equation}
Since $f_3\in C^{2,\alpha}(\Omega),$ we can apply the dominated convergence theorem to conclude that
\begin{equation}\label{citt1}
\lim_{\epsilon\to 0}\,\int_{B_{\epsilon}(z_1)\times D_2\times D_3}\frac{\partial^2 f_3}{\partial\bar{\zeta}_1\partial\bar{\zeta}_2}\,\frac{(\overline{\zeta_1-z_1})(\overline{\zeta_2-z_2})}{(\zeta_3-z_3)\,G}\, d\bar{\zeta}_1\wedge\dots\wedge d\zeta_3= 0.
\end{equation}
Note that \eqref{g} implies
\begin{equation*}
\bigg{|}\int_{\partial B_{\epsilon}(z_1)\times D_2\times D_3}\frac{(\overline{\zeta_1-z_1})(\overline{\zeta_2-z_2})}{(\zeta_3-z_3)\,G}\,d\zeta_1\wedge\dots\wedge d\zeta_3\bigg{|}\leq \int_{\partial B_{\epsilon}(z_1)\times D_2\times D_3} \frac{|d\zeta_1\wedge\dots\wedge d\zeta_3|}{|\zeta_1-z_1|^{\frac{3}{4}}|\zeta_2-z_2|^{\frac{3}{4}}|\zeta_3-z_3|^{\frac{3}{2}}}=O(\epsilon^{\frac{5}{4}}).
\end{equation*}
Hence
\begin{equation}\label{citt2}
\lim_{\epsilon\to 0}\,\int_{\partial B_{\epsilon}(z_1)\times D_2\times D_3}\frac{\partial f_3}{\partial\bar{\zeta}_2}\,\frac{(\overline{\zeta_1-z_1})(\overline{\zeta_2-z_2})}{(\zeta_3-z_3)\,G}\,d\zeta_1\wedge d\bar{\zeta}_2\wedge\dots\wedge d\zeta_3=0
\end{equation}
Equations \eqref{citt1} and \eqref{citt2} show that letting $\epsilon\to 0$ in \eqref{again} we obtain
\begin{equation}\label{againaa}
\begin{split}
I_3&=\int_{\partial D_1\times D_2\times D_3}\frac{\partial f_3}{\partial\bar{\zeta}_2}\,\frac{(\overline{\zeta_1-z_1})(\overline{\zeta_2-z_2})}{(\zeta_3-z_3)\,G}\,d\zeta_1\wedge\dots\wedge d\zeta_3\\&-\int_{D_1\times D_2\times D_3}\frac{\partial f_3}{\partial\bar{\zeta}_2}\frac{\partial}{\partial\bar{\zeta}_1}\bigg{[}\frac{(\overline{\zeta_1-z_1})(\overline{\zeta_2-z_2})}{(\zeta_3-z_3)\,G}\bigg{]}d\bar{\zeta}_1\wedge\dots\wedge d\zeta_3\\&=I_4-I_5.
\end{split}
\end{equation}
We call $I_4$ and $I_5$ the two integrals appearing on the right side of \eqref{againaa}. For each of them we
follow the same steps as above. That is, we remove a small ball $B_{\epsilon}(z_2)$ from $D_2$ and we apply Stokes' theorem. For $I_4$, we get
\begin{equation}\label{again2}
\begin{split}
I_4&=\int_{D_3}\Bigg{\{}\int_{\partial D_1\times \partial D_2} f_3\,\frac{(\overline{\zeta_1-z_1})(\overline{\zeta_2-z_2})}{(\zeta_3-z_3)\,G}\,d\zeta_1\wedge d\zeta_2 - \int_{\partial D_1\times D_2\setminus B_{\epsilon}(z_2)} f_3\,\frac{\partial}{\partial\bar{\zeta}_2}\bigg{[}\frac{(\overline{\zeta_1-z_1})(\overline{\zeta_2-z_2})}{(\zeta_3-z_3)\,G}\bigg{]}\,d\zeta_1\dots d\zeta_2\\&
+\int_{\partial D_1\times \partial B_{\epsilon}(z_2)} f_3\,\frac{(\overline{\zeta_1-z_1})(\overline{\zeta_2-z_2})}{(\zeta_3-z_3)\,G}\,d\zeta_1\wedge d\zeta_2+\int_{\partial D_1\times B_{\epsilon}(z_2)}\frac{\partial f_3}{\partial\bar{\zeta}_2}\,\frac{(\overline{\zeta_1-z_1})(\overline{\zeta_2-z_2})}{(\zeta_3-z_3)\,G}\,d\zeta_1\dots d\zeta_2\Bigg{\}}\, d\bar{\zeta}_3\wedge d\zeta_3.\\
\end{split}
\end{equation}
We argue in the same way as for \eqref{again}. Exploiting \eqref{g}, Lemma \ref{lemma1}, Lemma \ref{lemma2} and the dominated convergence theorem, we see that taking $\lim_{\epsilon\to 0}$ in \eqref{again2} gives
\begin{equation}\label{again2aa}
\begin{split}
I_4&=\int_{\partial D_1\times \partial D_2\times D_3} f_3\,\frac{(\overline{\zeta_1-z_1})(\overline{\zeta_2-z_2})}{(\zeta_3-z_3)\,G}\,d\zeta_1\wedge\dots\wedge d\zeta_3 \\&- \int_{\partial D_1\times D_2\times D_3} f_3\,\frac{\partial}{\partial\bar{\zeta}_2}\bigg{[}\frac{(\overline{\zeta_1-z_1})(\overline{\zeta_2-z_2})}{(\zeta_3-z_3)\,G}\bigg{]}\,d\zeta_1\wedge\dots\wedge d\zeta_3.\\
\end{split}
\end{equation} 
 For $I_5$, Stokes' theorem yields
 \begin{equation}\label{again3}
\begin{split}
I_5&=\int_{D_3}\Bigg{\{}\int_{ D_1\times \partial D_2} f_3\,\frac{\partial}{\partial\bar{\zeta}_1}\bigg{[}\frac{(\overline{\zeta_1-z_1})(\overline{\zeta_2-z_2})}{(\zeta_3-z_3)\,G}\bigg{]}\,d\bar{\zeta}_1\wedge d\zeta_{1}\wedge d\zeta_2 \\&- \int_{ D_1\times D_2\setminus B_{\epsilon}(z_2)} f_3\,\frac{\partial^2}{\partial\bar{\zeta}_1\partial\bar{\zeta}_2}\bigg{[}\frac{(\overline{\zeta_1-z_1})(\overline{\zeta_2-z_2})}{(\zeta_3-z_3)\,G}\bigg{]}\,d\bar{\zeta_1}\wedge d\zeta_1\wedge d\bar{\zeta}_2\wedge d\zeta_2\\&
+\int_{ D_1\times \partial B_{\epsilon}(z_2)} f_3\,\frac{\partial}{\partial\bar{\zeta}_1}\bigg{[}\frac{(\overline{\zeta_1-z_1})(\overline{\zeta_2-z_2})}{(\zeta_3-z_3)\,G}\bigg{]}\,d\bar{\zeta}_1\wedge d\zeta_1\wedge  d\zeta_2\\&+\int_{ D_1\times B_{\epsilon}(z_2)}\frac{\partial f_3}{\partial\bar{\zeta}_2}\,\frac{\partial}{\partial\bar{\zeta}_1}\bigg{[}\frac{(\overline{\zeta_1-z_1})(\overline{\zeta_2-z_2})}{(\zeta_3-z_3)\,G}\bigg{]}\,\,d\bar{\zeta_1}\wedge d\zeta_1\wedge d\bar{\zeta}_2\wedge d\zeta_2\Bigg{\}}\, d\bar{\zeta}_3\wedge d\zeta_3.\\
\end{split}
\end{equation}
Note that 
\begin{equation}\label{ggg}
\frac{\partial}{\partial\bar{\zeta}_1}\bigg{[}\frac{(\overline{\zeta_1-z_1})(\overline{\zeta_2-z_2})}{(\zeta_3-z_3)\,G}\bigg{]}=\frac{(\overline{\zeta_2-z_2})|\zeta_2-z_2|^2(\overline{\zeta_3-z_3})}{G^2}.
\end{equation}
From \eqref{theG3} with $k_1=9, k_2=1, k_3=6$, we get 
\begin{equation}\label{hhh}
\bigg{|}\frac{(\overline{\zeta_2-z_2})|\zeta_2-z_2|^2(\overline{\zeta_3-z_3})}{G^2}\bigg{|}\leq \frac{|\zeta_2-z_2|^3|\zeta_3-z_3|}{|\zeta_1-z_1|^{\frac{7}{4}}|\zeta_2-z_2|^{\frac{15}{4}}|\zeta_3-z_3|^{\frac{5}{2}}}=\frac{1}{|\zeta_1-z_1|^{\frac{7}{4}}|\zeta_2-z_2|^{\frac{3}{4}}|\zeta_3-z_3|^{\frac{3}{2}}}.
\end{equation}
Applying Lemma \ref{lemma1} three times, we see that there exists a constant $C$ such that 
\begin{equation}\label{hfinta}
\int_{D_1\times D_2\times D_3} \frac{|d\bar{\zeta}_1\wedge\dots\wedge d\zeta_3|}{|\zeta_1-z_1|^{\frac{7}{4}}|\zeta_1-z_1|^{\frac{3}{4}}|\zeta_1-z_1|^{\frac{3}{2}}} \leq C.
\end{equation}
We can thus apply the dominated convergence theorem to conclude that
\begin{equation*} 
\lim_{\epsilon\to 0}\,\int_{ D_1\times B_{\epsilon}(z_2)\times D_3}\frac{\partial f_3}{\partial\bar{\zeta}_2}\,\frac{\partial}{\partial\bar{\zeta}_1}\bigg{[}\frac{(\overline{\zeta_1-z_1})(\overline{\zeta_2-z_2})}{(\zeta_3-z_3)\,G}\bigg{]}\,\,d\bar{\zeta_1}\wedge d\zeta_1\wedge d\bar{\zeta}_2\wedge d\zeta_2\wedge d\bar{\zeta}_3\wedge d\zeta_3=0.
\end{equation*}
Moreover, equations \eqref{ggg} and \eqref{hhh} imply
\begin{equation*}
\begin{split}
\bigg{|}\int_{ D_1\times \partial B_{\epsilon}(z_2)\times D_3} \frac{\partial}{\partial\bar{\zeta}_1}\bigg{[}\frac{(\overline{\zeta_1-z_1})(\overline{\zeta_2-z_2})}{(\zeta_3-z_3)\,G}\bigg{]}\,d\bar{\zeta}_1\dots d\zeta_3\bigg{|}&\leq\int_{D_1\times  \partial B_{\epsilon}(z_2)\times D_3} \frac{|d\bar{\zeta}_1\wedge\dots\wedge d\zeta_3|}{|\zeta_1-z_1|^{\frac{7}{4}}|\zeta_2-z_2|^{\frac{3}{4}}|\zeta_3-z_3|^{\frac{3}{2}}}\\&=O(\epsilon^{\frac{5}{4}}). 
\end{split}
\end{equation*}
 Letting $\epsilon\to 0$ in \eqref{again3}, we therefore obtain
\begin{equation}\label{again33}
\begin{split}
I_5&=\int_{ D_1\times \partial D_2\times D_3} f_3\,\frac{\partial}{\partial\bar{\zeta}_1}\bigg{[}\frac{(\overline{\zeta_1-z_1})(\overline{\zeta_2-z_2})}{(\zeta_3-z_3)\,G}\bigg{]}\,d\bar{\zeta}_1\wedge d\zeta_{1}\wedge d\zeta_2\wedge d\bar{\zeta}_3\wedge d\zeta_3 \\&- \int_{ D_1\times D_2\times D_3} f_3\,\frac{\partial^2}{\partial\bar{\zeta}_1\partial\bar{\zeta}_2}\bigg{[}\frac{(\overline{\zeta_1-z_1})(\overline{\zeta_2-z_2})}{(\zeta_3-z_3)\,G}\bigg{]}\,d\bar{\zeta_1}\wedge d\zeta_1\wedge d\bar{\zeta}_2\wedge d\zeta_2\wedge d\bar{\zeta}_3\wedge d\zeta_3.
\end{split}
\end{equation}
Combining \eqref{again2aa} and \eqref{again33}, and recalling that $I_3=I_4-I_5$, we can write
\begin{equation}\label{toestt}
\begin{split}
I_3&=\int_{\partial D_1\times \partial D_2\times D_3} f_3\,\frac{(\overline{\zeta_1-z_1})(\overline{\zeta_2-z_2})}{(\zeta_3-z_3)\,G}\,d\zeta_1\wedge d\zeta_2\wedge d\bar{\zeta}_3\wedge d\zeta_3\\&-\int_{\partial D_1\times D_2\times D_3} f_3\,\frac{(\overline{\zeta_1-z_1})|\zeta_1-z_1|^2(\overline{\zeta_3-z_3})}{G^2}\,d\zeta_1\wedge  d\bar{\zeta}_2\wedge d\zeta_2\wedge d\bar{\zeta}_3\wedge d\zeta_3\\&-\int_{D_1\times \partial D_2\times D_3} f_3\,\frac{(\overline{\zeta_2-z_2})|\zeta_2-z_2|^2|(\overline{\zeta_3-z_3})}{G^2}\,  d\bar{\zeta}_1\wedge d\zeta_1\wedge d\zeta_2\wedge d\bar{\zeta}_3\wedge d\zeta_3\\&+\int_{D_1\times D_2\times D_3} f_3\,\frac{(\overline{\zeta_3-z_3})|\zeta_1-z_1|^2|\zeta_2-z_2|^2|\zeta_3-z_3|^2}{G^3}\,d\bar{\zeta}_1\wedge d\zeta_1\wedge d\bar{\zeta}_2\wedge d\zeta_2\wedge d\bar{\zeta}_3\wedge d\zeta_3.
\end{split}
\end{equation}
By \eqref{g}, Lemma \ref{lemma1} and Lemma \ref{lemma2}, there exists a constant $C$ such that
\begin{equation}\label{una}
\norm{\int_{\partial D_1\times \partial D_2\times D_3}\frac{(\overline{\zeta_1-z_1})(\overline{\zeta_2-z_2})}{(\zeta_3-z_3)\,G}\,d\zeta_1\wedge d\zeta_2\wedge d\bar{\zeta}_3\wedge d\zeta_3}_{L^{\infty}(\Omega)}\leq C.
\end{equation}
Similarly, by \eqref{hhh}, Lemma \ref{lemma1} and Lemma \ref{lemma2}, there exists a constant $C$ such that
\begin{equation}\label{due}
\norm{\int_{D_1\times \partial D_2\times D_3} \frac{(\overline{\zeta_2-z_2})|\zeta_2-z_2|^2|(\overline{\zeta_3-z_3})}{G^2}\, d\bar{\zeta}_1\wedge d\zeta_1\wedge d\zeta_2\wedge d\bar{\zeta}_3\wedge d\zeta_3}_{L^{\infty}(\Omega)}\leq C.
\end{equation}
From \eqref{theG3} applied with $k_1=k_2=9, k_3=6$, we obtain 
\begin{equation*}
\bigg{|}\frac{(\overline{\zeta_3-z_3})|\zeta_1-z_1|^2|\zeta_2-z_2|^2|\zeta_3-z_3|^2}{G^3}\bigg{|}\leq \frac{|\zeta_1-z_1|^2|\zeta_2-z_2|^2|\zeta_3-z_3|^3}{|\zeta_1-z_1|^{\frac{15}{4}}|\zeta_2-z_2|^{\frac{15}{4}}|\zeta_3-z_3|^{\frac{9}{2}}}=\frac{1}{|\zeta_1-z_1|^{\frac{7}{4}}|\zeta_1-z_1|^{\frac{7}{4}}|\zeta_1-z_1|^{\frac{3}{2}}}.
\end{equation*}
Hence, by Lemma \ref{lemma1} and Lemma \ref{lemma2}, there exists a constant $C$ such that
\begin{equation}\label{tre}
\norm{\int_{D_1\times D_2\times D_3} \frac{(\overline{\zeta_3-z_3})|\zeta_1-z_1|^2|\zeta_2-z_2|^2|\zeta_3-z_3|^2}{G^3}\,d\bar{\zeta}_1\wedge d\zeta_1\wedge d\bar{\zeta}_2\wedge d\zeta_2\wedge d\bar{\zeta}_3\wedge d\zeta_3}_{L^{\infty}(\Omega)}\leq C.
\end{equation}
Now observe that the second and third integral on the right side of \eqref{toestt} can be estimated in the same way by reversing the roles of the two variables. From \eqref{una}, \eqref{due} and \eqref{tre}, we can thus conclude that there exists a constant $C$ such that
\begin{equation*}
\norm{I_3}_{L^{\infty}(\Omega)}\leq C\norm{f}_{L^{\infty}(\Omega)}.
\end{equation*}
This proves \eqref{goaaaal}.
\end{proof}

\section{The general dimension case}\label{sectionCn}
For $n\geq 2$, let $D_1,\dots,D_n\subset\C$ be open bounded domains with $C^{1,\alpha}$ boundary, where $0<\alpha<1$. Consider the product domain $\Omega=D_1\times\dots\times D_n$. For a $(0,1)$ form $f=f_1d\bar{z}_1+\dots+f_nd\bar{z}_n$ in $\Omega$ with components $f_j\in C^{n-1,\alpha}(\Omega)$, we define the integral operator ${\bf T}f$ as follows:
\begin{equation}\label{solutionformula}
{\bf T}f:=\sum_{s=1}^{n} (-1)^{s-1}\sum_{1\leq i_{1}<\dots<i_s\leq n}T^{i_1}\dots T^{i_s}\bigg{(}\frac{\partial^{s-1}f_{i_s}}{\partial \bar{z}_{i_1}\dots\partial\bar{z}_{i_{s-1}}}\bigg{)}.
\end{equation}
Note that the regularity of the $f_j$ ensures that ${\bf T}f$ is well defined. 
\begin{lemma}\label{generalsolutionlemma}
${\bf T}f\in C^{1,\alpha}(\Omega)$. Moreover, if the form $f$ is $\bar\partial$-closed, then $\bar\partial {\bf T}f=f$ in $\Omega$.
\end{lemma}
\begin{proof}
The proof of the first statement follows the same steps as the two-dimensional case (Lemma \ref{P}). Now fix $k\in\{1,\dots,n\}$. We want to prove that $\partial_{\bar{z}_k}{\bf T}f=f_k$. First we rewrite the formula for ${\bf T}f$ separating the sets of indices $I=\{1\leq i_1<\dots<i_s\leq n\}$ for which $k\in I$. Hence
\begin{equation}\label{toberewritten}
\begin{split}
{\bf T}f&=\sum_{s=1}^{n} (-1)^{s-1}\sum_{\substack{I=\{1\leq i_{1}<\dots<i_s\leq n\}\\k\in I}}T^{i_1}\dots T^{i_s}\bigg{(}\frac{\partial^{s-1}f_{i_s}}{\partial \bar{z}_{i_1}\dots\partial\bar{z}_{i_{s-1}}}\bigg{)}\\&+\sum_{s=1}^{n-1} (-1)^{s-1}\sum_{\substack{I=\{1\leq i_{1}<\dots<i_s\leq n\}\\ k\not\in I}}T^{i_1}\dots T^{i_s}\bigg{(}\frac{\partial^{s-1}f_{i_s}}{\partial \bar{z}_{i_1}\dots\partial\bar{z}_{i_{s-1}}}\bigg{)}.
\end{split}
\end{equation}
Recall the following two facts:\begin{itemize}
\item Since $f$ is $\bar\partial$-closed, then for every set $I=\{1\leq i_1<\dots<i_s\leq n\}$ and every $k\in\{1,\dots,n\}$, we have
\begin{equation*}
\frac{\partial^s f_k}{\partial\bar{z}_{i_1}\dots\partial\bar{z}_{i_{s}}}=\frac{\partial^s f_{i_{s}}}{\partial\bar{z}_k\partial\bar{z}_{i_1}\dots\partial\bar{z}_{i_{s-1}}}.
\end{equation*}
\item By Fubini's theorem, the operators $T^i$ commute. That is, $T^iT^j=T^jT^i$ for all $i,j\in\{1,\dots,n\}$.
\end{itemize}
We can therefore rewrite \eqref{toberewritten} as
\begin{equation}\label{toapply}
\begin{split}
{\bf T}f=T^kf_k&+\sum_{s=1}^{n-1}(-1)^s \sum_{\substack{I=\{1\leq i_{1}<\dots<i_s\leq n\}\\k\not\in I}}T^kT^{i_1}\dots T^{i_s}\bigg{(}\frac{\partial^{s}f_{i_s}}{\partial\bar{z}_k\partial\bar{z}_{i_1}\dots\partial\bar{z}_{i_{s-1}}}\bigg{)}\\&+\sum_{s=1}^{n-1}(-1)^{s-1}\sum_{\substack{I=\{1\leq i_{1}<\dots<i_s\leq n\}\\k\not\in I}}T^{i_1}\dots T^{i_s}\bigg{(}\frac{\partial^{s-1}f_{i_s}}{\partial \bar{z}_{i_1}\dots\partial\bar{z}_{i_{s-1}}}\bigg{)}.
\end{split}
\end{equation}
Recall that if $g\in C^{1,\alpha}(\Omega)$, then $\partial_{\bar{z}_k}T^kg=g$ in $\Omega$. Hence, applying the operator $\partial_{\bar{z}_k}$ to \eqref{toapply}, we obtain 
\begin{equation}\label{weobtain}
\begin{split}
\partial_{\bar{z}_k}{\bf T}f=f_k&+\sum_{s=1}^{n-1}(-1)^s \sum_{\substack{I=\{1\leq i_{1}<\dots<i_s\leq n\}\\k\not\in I}}T^{i_1}\dots T^{i_s}\bigg{(}\frac{\partial^{s}f_{i_s}}{\partial\bar{z}_k\partial\bar{z}_{i_1}\dots\partial\bar{z}_{i_{s-1}}}\bigg{)}\\&+\sum_{s=1}^{n-1}(-1)^{s-1}\sum_{\substack{I=\{1\leq i_{1}<\dots<i_s\leq n\}\\k\not\in I}}T^{i_1}\dots T^{i_s}\bigg{(}\frac{\partial^{s}f_{i_s}}{\partial\bar{z}_k\partial \bar{z}_{i_1}\dots\partial\bar{z}_{i_{s-1}}}\bigg{)}.
\end{split}
\end{equation}
Since the two sums on the right side of \eqref{weobtain} cancel, then $\partial_{\bar{z}_k}{\bf T}f=f_k$, and the proof is complete.
\end{proof}
\begin{theorem}\label{main}
Let $D_1,\dots,D_n\subset\C$ be open bounded domains with $C^{1,\alpha}$ boundary, where $0<\alpha<1$. Consider the product domain $\Omega=D_1\times\dots\times D_n$. Let $f=f_1d\bar{z}_1+\dots+f_nd\bar{z}_n$ be a $\bar\partial$-closed $(0,1)$ form on $\Omega$ with components $f_j\in C^{n-1,\alpha}(\Omega).$ Then ${\bf T}f\in C^{1,\alpha}(\Omega)$ and $\bar\partial {\bf T}f=f$ in $\Omega$. Moreover, ${\bf T}f$ satisfies the supnorm estimate
\begin{equation*}
\norm{{\bf T}f}_{L^{\infty}(\Omega)}\leq C\norm{f}_{L^{\infty}(\Omega)}
\end{equation*}
for some constant $C$ independent of $f$.
\end{theorem}
\begin{remark}
If $\Omega$ is a polydisc, then by rescaling the variables, considering for $\epsilon>0$ the form of components $f^{\epsilon}_j=f_j((1-\epsilon)z)$, we can drop the regularity assumption on the components $f_j$ from $C^{n-1,\alpha}(\Omega)$ to $C^{n-1}(\Omega)$.
\end{remark}

The proof of Theorem \ref{main} follows the same steps as in dimension two and three.
We start by proving a general version of the decomposition already exploited in \eqref{decompositionC2} and \eqref{decompositionC3}.
\begin{lemma}\label{decompositioncomplexnumbers}
Given non-zero complex numbers $a_1,\dots,a_m\in \C$, let $G:=\sum_{k=1}^m\prod_{l=1,l\neq k}^m |a_l|^2.$ Then
\begin{equation}\label{equationeasylemmadecomposition}
\frac{1}{a_1\cdots a_m}=\sum_{k=1}^m\frac{1}{a_kG}\prod _{\substack{l=1\\l\neq k}}^m\bar{a}_l.
\end{equation}
\end{lemma}
\begin{proof}
Taking common denominator on the right side of \eqref{equationeasylemmadecomposition}, we obtain
\begin{equation*}
\sum_{k=1}^m\frac{1}{a_kG}\prod _{\substack{l=1\\l\neq k}}^m\bar{a}_l=\sum_{k=1}^m\frac{\Big{(}\prod_{l=1,l\neq k} a_l\Big{)}\Big{(}\prod_{l=1,l\neq k}\bar{a}_l\Big{)}}{a_1\cdots a_m\cdot G}=\frac{1}{a_1\cdots a_m}.
\end{equation*}
The lemma is thus proved.
\end{proof}
The next lemma gives a formula for iterated applications of Stokes' theorem on a product domain in $\C^n$.
\begin{lemma}\label{lemmaStokesgeneral}
Let $D_1,\dots,D_n$ be open bounded domains in $\C$ with $C^1$ boundary. Let $\Omega=D_1\times\dots\times D_n$ and $f,g\in C^n(\overline{\Omega})$. Then 
\begin{equation*}\label{stokes}
\begin{split}
&\int_{D_1\times\dots\times D_n}\frac{\partial^{n}f}{\partial\bar{\zeta}_1\dots\partial\bar{\zeta}_{n}}\, g\,d\bar{\zeta}_1\wedge\dots\wedge d\zeta_{n}=\\&\sum_{m=0}^n (-1)^m\sum_{1\leq i_1<\dots<i_m\leq n}\int_{D_{i_1}\times\dots\times D_{i_m}\times\partial D_{j_1}\times\dots\times\partial D_{j_{n-m}}}f\,\frac{\partial^m g}{\partial \bar{\zeta}_{i_1}\dots\partial\bar{\zeta}_{i_m}}\,d\bar{\zeta}_{i_1}\wedge\dots\wedge d\zeta_{i_m}\wedge d\zeta_{j_1}\wedge\dots\wedge d\zeta_{j_{n-m}}.
\end{split}
\end{equation*}
Here $\{i_1,\dots,i_{m}\}\cup\{j_1,\dots,j_{n-m}\}=\{1,\dots,n\}$.
\end{lemma}
\begin{proof}
We argue by induction on $n$. The case $n=1$ is the usual formula of integration by parts. Assume now that \eqref{stokes} holds for the product of $n-1$ domains. Then
\begin{equation*}
\begin{split}
&\int_{D_1\times\dots\times D_n}\frac{\partial^{n}f}{\partial\bar{\zeta}_1\dots\partial\bar{\zeta}_{n}}\, g\,d\bar{\zeta}_1\wedge\dots\wedge d\zeta_{n}=\\&\sum_{m=0}^{n-1}(-1)^m\sum_{1\leq i_1<\dots<i_m\leq n-1}\int_{D_n}\int_{D_{i_1}\times\dots\times D_{i_m}\times\partial D_{j_1}\times\dots\times\partial D_{j_{n-m-1}}}\frac{\partial f}{\partial\bar{\zeta}_n}\,\frac{\partial^m g}{\partial \bar{\zeta}_{i_1}\dots\partial\bar{\zeta}_{i_m}}\,d\bar{\zeta}_{i_1}\dots d\zeta_{j_{n-m-1}}\,d\bar{\zeta}_{n}\wedge d\zeta_n.
\end{split}
\end{equation*}
Applying Stokes' theorem in the variable $\zeta_n$ we get
\begin{equation*}
\begin{split}
&\int_{D_1\times\dots\times D_n}\frac{\partial^{n}f}{\partial\bar{\zeta}_1\dots\partial\bar{\zeta}_{n}}\, g\,d\bar{\zeta}_1\wedge\dots\wedge d\zeta_{n}=\\&\sum_{m=0}^{n-1}(-1)^m\sum_{1\leq i_1<\dots<i_m\leq n-1}\int_{D_{i_1}\times\dots\times D_{i_m}\times\partial D_{j_1}\times\dots\times\partial D_{j_{n-m-1}}\times\partial D_n}f\,\frac{\partial^m g}{\partial \bar{\zeta}_{i_1}\dots\partial\bar{\zeta}_{i_m}}\,d\bar{\zeta}_{i_1}\wedge\dots\wedge d\zeta_{j_{n-m-1}}\wedge d\zeta_n
\\&\sum_{m=0}^{n-1}(-1)^{m+1}\sum_{1\leq i_1<\dots<i_m\leq n-1}\int_{D_{i_1}\times\dots\times D_{i_m}\times D_n\times\partial D_{j_1}\times\dots\times\partial D_{j_{n-m-1}}}f\,\frac{\partial^{m+1} g}{\partial \bar{\zeta}_{i_1}\dots\partial\bar{\zeta}_{i_m}\partial\bar{\zeta}_n}\,d\bar{\zeta}_{i_1}\wedge\dots\wedge d\zeta_{j_{n-m-1}}.
\end{split}
\end{equation*}
Re-indexing the sums, we obtain
\begin{equation*}
\begin{split}
&\int_{D_1\times\dots\times D_n}\frac{\partial^{n}f}{\partial\bar{\zeta}_1\dots\partial\bar{\zeta}_{n}}\, g\,d\bar{\zeta}_1\wedge\dots\wedge d\zeta_{n}=\\& \sum_{m=0}^{n-1}(-1)^m\sum_{1\leq i_1<\dots<i_m\leq n-1}\int_{D_{i_1}\times\dots D_{i_m}\times\partial D_{j_1}\times\dots\times\partial D_{j_{n-m}}}f\,\frac{\partial^m g}{\partial \bar{\zeta}_{i_1}\dots\partial\bar{\zeta}_{i_m}}\,d\bar{\zeta}_{i_1}\wedge\dots\wedge d\zeta_{i_m}\wedge d\zeta_{j_1}\wedge\dots\wedge d\zeta_{j_{n-m}}\\&\sum_{m=1}^{n}(-1)^m\sum_{1\leq i_1<\dots<i_m=n}\int_{D_{i_1}\times\dots D_{i_m}\times\partial D_{j_1}\times\dots\times\partial D_{j_{n-m-1}}}f\,\frac{\partial^{m} g}{\partial \bar{\zeta}_{i_1}\dots\partial\bar{\zeta}_{i_m}}\,d\bar{\zeta}_{i_1}\wedge\dots\wedge d\zeta_{i_m}\wedge d\zeta_{j_1}\wedge\dots\wedge d\zeta_{j_{n-m}},\\
\end{split}
\end{equation*}
which completes the proof of the lemma.
\end{proof} 

\begin{proof}[Proof of Theorem \ref{main}]
We only need to prove that the estimate $\norm{Tf}_{L^{\infty}(\Omega)}\leq C\norm{f}_{L^{\infty}(\Omega)}$ holds for some constant $C$ independent of $f$. We argue by induction on the dimension $n$. Note that the cases $n=2$ and $n=3$ have already been proved in Theorem \ref{theoremC2} and Theorem \ref{theoremC3} respectively. Now assume that the theorem holds for $n-1$. Then in \eqref{solutionformula} all the terms where $s\leq n-1$ can be estimated, and we just have to argue for the expression
\begin{equation*}
T^1\dots T^n\bigg{(}\frac{\partial^{n-1}f_n}{\partial\bar{z}_1\dots\partial\bar{z}_{n-1}}\bigg{)}=\frac{1}{(-2\pi i)^n}\int_{D_1\times\dots\times D_n}\frac{\partial^{n-1}f_n}{\partial\bar{\zeta}_1\dots\partial\bar{\zeta}_{n-1}}\frac{1}{(\overline{\zeta_1-z_1})\dots(\overline{\zeta_n-z_n}))}\,d\bar{\zeta}_1\wedge\dots\wedge d\zeta_n.
\end{equation*}
For convenience, define
\begin{equation*}
Df:=\frac{\partial^{n-1}f_n}{\partial\bar{z}_1\dots\partial\bar{z}_{n-1}}.
\end{equation*}
 By Lemma \ref{decompositioncomplexnumbers} applied with $a_j=(\zeta_j-z_j)$ and $G=\sum_{k=1}^n\prod_{l=1,l\neq k}^n |a_l|^2$, we can write 
\begin{equation}\label{thesum}
\begin{split}
(-2\pi i)^nT^1T^2\dots T^n (Df)&=\sum_{k=1}^n\int_{D_1\times\dots\times D_n} Df\,\frac{\prod_{l=1,l\neq k}^n(\overline{\zeta_l-z_l})}{(\zeta_k-z_k)\,G}\,d\bar{\zeta}_1\wedge\dots\wedge d\zeta_n.\\
\end{split}
\end{equation}
Rewriting $Df$ in $n$ equivalent ways (exploiting that $f$ is $\bar\partial$-closed), we can expand the sum in \eqref{thesum} as
\begin{equation}\label{in}
(-2\pi i)^nT^1T^2\dots T^n (Df)=\sum_{k=1}^n\int_{D_1\times\dots\times D_n}\frac{\partial^{n-1}f_k}{\partial\bar{\zeta}_1\dots\widehat{\partial\bar{\zeta}_k}\dots\partial\bar{\zeta}_n}\frac{\prod_{l=1,l\neq k}^n(\overline{\zeta_l-z_l})}{(\zeta_k-z_k)\,G}\,d\bar{\zeta}_1\wedge\dots\wedge d\zeta_n.
\end{equation}
Here the notation $\widehat{\partial\bar{\zeta}_k}$ indicates that the corresponding term has been removed. Note that the integrals appearing in \eqref{in} can all be estimated in the same way, by renaming the variables. We can therefore restrict our attention to one of them, say the one where $k=n$. The theorem is thus proved if we can estimate the integral
\begin{equation}\label{te}
\int_{D_1\times\dots\times D_n}\frac{\partial^{n-1}f_n}{\partial\bar{\zeta}_1\dots\partial\bar{\zeta}_{n-1}}\frac{(\overline{\zeta_1-z_1})\dots(\overline{\zeta_{n-1}-z_{n-1}})}{(\zeta_n-z_n)\,G}\,d\bar{\zeta}_1\wedge\dots\wedge d\zeta_n.
\end{equation}
For fixed $z=(z_1,\dots,z_n)\in D_1\times\dots\times D_n$, we define \[g_z:=\frac{(\overline{\zeta_1-z_1})\dots(\overline{\zeta_{n-1}-z_{n-1}})}{(\zeta_n-z_n)\,G}.\]
Although $g_z\not\in C^n(\overline{\Omega})$, Lemma \ref{lemmaStokesgeneral} still holds. In fact, we can deal with the singularities of $g_z$ in the same way as for the cases of dimension two and three: at each application of Stokes' theorem, we remove a small ball of radius $\epsilon$ from the point of singularity and then apply the dominated convergence theorem, letting $\epsilon\to 0$. For the clarity of the exposition, we do not repeat such details here.

By Lemma \ref{lemmaStokesgeneral}, in order to estimate \eqref{te}, it is enough to prove, for every $m\in\{0,\dots,n-1\}$ and every choice of $m$ indices $1\leq i_1<\dots<i_m\leq n-1$, that there exists a constant $C$ such that
\begin{equation*}
\norm{\int_{D_{i_1}\times\dots\times D_{i_m}\times\partial D_{j_1}\times\dots\times\partial D_{j_{n-1-m}}\times D_n}\frac{\partial^m g_z}{\partial\bar{\zeta}_{i_1}\dots\partial\bar{\zeta}_{i_m}}\,d\bar{\zeta}_{i_1}\wedge\dots\wedge d\zeta_n}_{L^{\infty}(\Omega)}\leq C.
\end{equation*}
Here \[\{j_1,\dots,j_{n-1-m}\}\cup\{i_1,\dots,i_m\}=\{1,\dots,n-1\}.\] Note that $m=0$ corresponds to taking no derivatives, that is, showing that 
\begin{equation*}
\norm{\int_{\partial D_1\times\dots\times\partial D_{n-1}\times D_n} g_z\,d\zeta_{i_1}\wedge\dots\wedge d\zeta_n}_{L^{\infty}(\Omega)}\leq C.
\end{equation*}
Up to renaming the variables, because of the symmetries of the function $g_z$, it is enough to prove, for every $m\in\{0,\dots,n-1\}$, that there exists a constant $C$ such that
\begin{equation}\label{gol}
\norm{\int_{D_{1}\times\dots\times D_{m}\times\partial D_{m+1}\times\dots\times\partial D_{n-1}\times D_n}\frac{\partial^m g_z}{\partial\bar{\zeta}_{1}\dots\partial\bar{\zeta}_{m}}\,d\bar{\zeta}_{1}\wedge\dots\wedge d\zeta_n}_{L^{\infty}(\Omega)}\leq C.
\end{equation}
 Once again, we allow $m$ to be equal to $0$.
 
We first prove a formula for the derivatives of $g_z$. For every $m=1,\dots,n-1$ we have
\begin{equation}\label{derivativeofg}
\frac{\partial g_z}{\partial\bar{\zeta}_1\dots\partial\bar{\zeta}_m}=\frac{m!\big{(}\prod_{j=1}^m|\zeta_j-z_j|^{2(m-1)}\Big{)}\Big{(}\prod_{j=m+1}^{n-1}(\overline{\zeta_{j}-z_{j}})|\zeta_{j}-z_{j}|^{2m}\Big{)}(\overline{\zeta_n-z_n})|\zeta_n-z_n|^{2m-2}}{G^{m+1}}.
\end{equation}
The proof of \eqref{derivativeofg} is a simple induction on $m$. The case $m=1$ is clear. Assume that \eqref{derivativeofg} holds for some $m\leq n-2$. Then 
\begin{equation}\label{2}
\begin{split}
\frac{\partial g_z}{\partial\bar{\zeta}_1\dots\partial\bar{\zeta}_{m+1}}
&=m!\bigg{(}\prod_{j=1}^m|\zeta_j-z_j|^{2(m-1)}\bigg{)}(\zeta_{m+1}-z_{m+1})^m\bigg{(}\prod_{j=m+2}^{n-1}(\overline{\zeta_{j}-z_{j}})|\zeta_j-z_j|^{2m}\bigg{)}\\&(\overline{\zeta_n-z_n})|\zeta_n-z_n|^{2m-2}\frac{\partial}{\partial\bar{\zeta}_{m+1}}\Bigg{\{}\bigg{(}\frac{(\overline{\zeta_{m+1}-z_{m+1}})}{G}\bigg{)}^{m+1}\Bigg{\}}.
\end{split}
\end{equation}
It is easy to see that the right side of \eqref{2} is equal to the right side of \eqref{derivativeofg} with $m$ replaced by $m+1$. Now that we have proved \eqref{derivativeofg}, we use it to show \eqref{gol}. We follow the strategy already employed in the cases $n=2$ and $n=3$. Note that if $b_1,\dots, b_n$ are non-negative real numbers, $k_1,\dots,k_n$ are non-negative integers, and $k=k_1+\dots+k_n$, then
\begin{equation}\label{easy}
(b_1+ \dots +b_n)^k \,\geq \, b_1^{k_1} \cdots b_n^{k_n}.
\end{equation}
We write $G=\sum_{j=1}^nb_j$, with $b_j=\prod_{l=1,l\neq j}^n|\zeta_l-z_l|^2$. By \eqref{easy}, we have \begin{equation*}
G^{m+1}\geq b_1^{\frac{k_1(m+1)}{k}}\cdots b_n^{\frac{k_n(m+1)}{k}}.
\end{equation*} 
Hence \eqref{gol} is proved if for every $m\in\{0,\dots,n-1\}$ we can find positive integers $k=k_1+\dots+k_n$ and a constant $C$ such that
\begin{equation}\label{golla}
\norm{\int_{D_{1}\times\dots\times D_{m}\times\partial D_{m+1}\times\dots\times\partial D_{n-1}\times D_n}H_m\,d\bar{\zeta}_{1}\wedge\dots\wedge d\zeta_n}_{L^{\infty}(\Omega)}\leq C,
\end{equation}
where
\begin{equation}\label{derivativeofg2}
H_m=\frac{\big{(}\prod_{j=1}^m|\zeta_j-z_j|^{2(m-1)}\Big{)}\Big{(}\prod_{j=m+1}^{n-1}(\overline{\zeta_{j}-z_{j}})|\zeta_{j}-z_{j}|^{2m}\Big{)}(\overline{\zeta_n-z_n})|\zeta_n-z_n|^{2m-2}}{b_1^{\frac{k_1(m+1)}{k}}\cdots b_n^{\frac{k_n(m+1)}{k}}}.
\end{equation}
By the definition of the $b_j$, we can rewrite \eqref{derivativeofg2} as
\begin{equation}\label{last}
H_m=\frac{\big{(}\prod_{j=1}^m|\zeta_j-z_j|^{2(m-1)}\Big{)}\Big{(}\prod_{j=m+1}^{n-1}(\overline{\zeta_{j}-z_{j}})|\zeta_{j}-z_{j}|^{2m}\Big{)}(\overline{\zeta_n-z_n})|\zeta_n-z_n|^{2m-2}}{|\zeta_1-z_1|^{\frac{2(m+1)}{k}(k_2+\dots+k_n)}\cdots |\zeta_n-z_n|^{\frac{2(m+1)}{k}(k_1+\dots+k_{n-1})}}.
\end{equation}
By Lemma \ref{lemma1} and Lemma \ref{lemma2}, inequality \eqref{golla} holds for some constant $C$ if the numbers $k,k_1,\dots,k_n$ satisfy the system 
\begin{equation}\label{sist}
\begin{cases}
 \frac{2(m+1)}{k}(k-k_j)-2(m-1)<2\quad j=1,\dots,m\\
\frac{2(m+1)}{k}(k-k_j)-(2m+1)<1\quad j=m+1,\dots,n-1\\
\frac{2(m+1)}{k}(k-k_n)-(2m-1)<2\\
k_1+\dots+k_n=k.
\end{cases}
\end{equation} 
Expanding, we see that \eqref{sist} is equivalent to
\begin{equation}\label{sist2}
\begin{cases}
k_j>\frac{k}{m+1}\quad j=1,\dots,m\\
k_j>0\quad j=m+1,\dots,n-1\\
k_n>\frac{k}{2(m+1)}\\
k_1+\dots+k_n=k.
\end{cases}
\end{equation}
One can easily verify that letting
\begin{equation}\label{sist3}
\begin{cases}
k=4(n-1)(m+1)\\
k_j=4(n-1)+1\quad j=1,\dots,m\\
k_j=1 \quad j=m+1,\dots,n-1\\
k_n=k-(k_1+\dots+k_{n-1}),
\end{cases}
\end{equation}
then the system \eqref{sist2} is satisfied for every choice of $n\geq 2$ and $m=0,\dots,n-1$. The proof of the theorem is therefore complete.
\end{proof}

\section{The integral operator ${\bf \widetilde{T}}$}\label{S6}
Let $\Omega=D_1\times\dots\times D_n$ be the product of one-dimensional open bounded domains $D_j$ with $C^{1,\alpha}$ boundary, where $0<\alpha<1$. Let $f=f_1d\bar{z}_1+\dots +f_nd\bar{z}_n$ be a $\bar\partial$-closed $(0,1)$ form with components $f_j\in C^{n-1,\alpha}(\Omega)$. In Section \ref{sectionCn} we showed that a solution to $\bar\partial u=f$ in $\Omega$ that satisfies supnorm estimates in $\Omega$ is given by $u={\bf T}f$, where
\begin{equation}\label{mit}
{\bf T}f:=\sum_{s=1}^{n} (-1)^{s-1}\sum_{1\leq i_{1}<\dots<i_s\leq n}T^{i_1}\dots T^{i_s}\bigg{(}\frac{\partial^{s-1}f_{i_s}}{\partial \bar{z}_{i_1}\dots\partial\bar{z}_{i_{s-1}}}\bigg{)}.
\end{equation}
Recall that for every set of indices $\{1\leq i_1<\dots <i_s\leq n\}$ we have
\begin{equation}\label{lan}
T^{i_1}\dots T^{i_s}\bigg{(}\frac{\partial^{s-1}f_{i_s}}{\partial \bar{z}_{i_1}\dots\partial\bar{z}_{i_{s-1}}}\bigg{)}(z_1,\dots,z_n)=\frac{1}{(-2\pi i)^s}\int_{D_{i_1}\times\dots\times D_{i_s}}\frac{\partial^{s-1}f_{i_s}}{\partial\bar{\zeta}_{i_1}\dots\partial\bar{\zeta}_{i_{s-1}}}\frac{1}{(\zeta_{i_1}-z_{i_1})\dots(\zeta_{i_s}-z_{i_s})}.
\end{equation}
\begin{notation}
To keep our formulas shorter, throughout this section we often omit the volume form when displaying an integral. 
\end{notation}

The formula for ${\bf T}f$ involves ``solid integrals" over the domains $D_j$ and derivatives up to order $n-1$ of the components $f_j$. The goal of this section is to make explicit something that has already been achieved in the proof of Theorem \ref{main}. Namely, we show how one can rewrite ${\bf T}f$ in such a way that no derivatives of the components $f_j$ appear in the formula, but paying the price of introducing integrals over the boundaries $\partial D_j$. 

We start by defining
\begin{equation*}
G_z^{i_1\dots i_s}:=\sum_{k=1}^s\prod_{\substack{l=1\\ l\neq k}}^s|\zeta_{i_l}-z_{i_l}|^2.
\end{equation*}
Exploiting that $f$ is $\bar\partial$-closed, we can rewrite \eqref{lan} as
\begin{equation*}
T^{i_1}\dots T^{i_s}\bigg{(}\frac{\partial^{s-1}f_{i_s}}{\partial \bar{z}_{i_1}\dots\partial\bar{z}_{i_{s-1}}}\bigg{)}(z_1,\dots,z_n)=\frac{1}{(-2\pi i)^s}\sum_{k=1}^s \int_{D_{i_1}\times\dots\times D_{i_s}}\frac{\partial^{s-1}f_{i_k}}{\partial\bar{\zeta}_{i_1}\dots\widehat{\partial\bar{\zeta}_k}\dots\partial\bar{\zeta}_{i_{s}}}\frac{\prod_{l=1, l\neq k}^s(\overline{\zeta_{i_l}-z_{i_l}})}{(\zeta_{i_k}-z_{i_k})\,G_z^{i_1\dots i_s}}.
\end{equation*}
For convenience, we define
\begin{equation*}
 g_{z}^{k,i_1\dots i_s}:=\frac{\prod_{l=1, l\neq k}^s(\overline{\zeta_{i_l}-z_{i_l}})}{(\zeta_{i_k}-z_{i_k})\,G_z^{i_1\dots i_s}}.
\end{equation*}
The proof of Theorem \ref{main} shows that by repeated applications of Stokes' theorem it is possible to move all the derivatives from the components $f_{j}$ to the kernel functions $g_z^{k,i_1\dots i_s}$. One thus obtains the following formula:
\begin{equation}\label{usoasdef}
\begin{split}
&T^{i_1}\dots T^{i_s}\bigg{(}\frac{\partial^{s-1}f_{i_s}}{\partial \bar{z}_{i_1}\dots\partial\bar{z}_{i_{s-1}}}\bigg{)}(z_1,\dots,z_n)=\\&\frac{1}{(-2\pi i)^s}\sum_{k=1}^s \sum_{m=0}^{s-1} (-1)^m\sum_{1\leq j_1<\dots <j_m\leq i_{s}}\int_{D_{j_1}\times\dots\times D_{j_m}\times\partial D_{t_1}\times\dots\times\partial D_{t_{s-m-1}}\times D_{i_k}}f_{i_k}(z',\zeta)\,\frac{\partial^m g_z^{k,i_1\dots i_s}}{\partial\bar{\zeta}_{j_1}\dots\partial\bar{\zeta}_{j_m}}.
\end{split}
\end{equation}
Here, for every $k$, the $j_l$ and the $t_l$ are two complementary sets of indices in $\{i_1,\dots,\widehat{i_k},\dots,i_s\}$. That is, \begin{equation}\label{ind}
\{j_1,\dots,j_m\}\cup\{t_1,\dots,t_{s-m-1}\}=\{i_1,\dots,\widehat{i_k},\dots,i_s\}.
\end{equation}
For each integral in \eqref{usoasdef}, we have denoted by $\zeta$ the variables upon which the integration occurs, and by $z'$ the components of the point $z=(z_1,\dots,z_n)$ in the variables that are not integrated. From now on we will omit this bit of notation whenever there is no ambiguity.

Equation \eqref{usoasdef} combined with \eqref{mit} achieves the goal described above: we now have a formula for ${\bf T}f$ which does not involve taking derivatives of the components $f_j$. We have introduced, however, integrals over the boundaries $\partial D_j$. 
\begin{example}
We show the explicit formula for ${\bf T}f$ that one obtains in the two-dimensional case following the steps described above.
\begin{equation}\label{dim2tilde}
\begin{split}
{\bf T}f&=T^1f_1+T^2f_2-\frac{1}{(2\pi i)^2}\,\Bigg{(}\int_{\partial D_1\times D_2}\frac{(\overline{\zeta_1-z_1}) f_2(\zeta_1,\zeta_2)}{(\zeta_2-z_2)|\zeta-z|^2}\,d\zeta_1\wedge d\bar{\zeta}_2\wedge d\zeta_2\\&+\int_{ D_1\times \partial D_2} \frac{(\overline{\zeta_2-z_2})f_1(\zeta_1,\zeta_2)}{(\zeta_1-z_1)|\zeta-z|^2}\,d\bar{\zeta}_1\wedge d\zeta_1\wedge d\zeta_2-\int_{ D_1\times D_2} \frac{(\overline{\zeta_1-z_1})f_1(\zeta_1,\zeta_2)}{|\zeta-z|^4}\,d\bar{\zeta}_1\wedge d\zeta_1 \wedge d\bar{\zeta}_2 \wedge d\zeta_2\\&-\int_{ D_1\times D_2} \frac{(\overline{\zeta_2-z_2})f_2(\zeta_1,\zeta_2)}{|\zeta-z|^4}\,d\bar{\zeta}_1\wedge d\zeta_1\wedge d\bar{\zeta}_2\wedge d\zeta_2\Bigg{)}.
\end{split}
\end{equation}
\end{example}
In the proof of Theorem \ref{main} we showed, for each choice of indices as in \eqref{ind}, that there exists a constant $C$ such that
\begin{equation}\label{kh}
\norm{\int_{D_{j_1}\times\dots\times D_{j_m}\times\partial D_{t_1}\times\dots\times\partial D_{t_{s-m-1}}\times D_{i_k}}\bigg{|}\frac{\partial^m g_z^{k,i_1\dots i_s}}{\partial\bar{\zeta}_{j_1}\dots\partial\bar{\zeta}_{j_m}}\bigg{|}}_{L^{\infty}(\Omega)}\leq C.
\end{equation}
Hence the right side of \eqref{usoasdef} makes sense even when the components $f_j$ are just assumed continuous in $\overline{\Omega}$. We can thus use \eqref{usoasdef} to define a new operator $T^{[i_1\dots i_s]}$ on $(0,1)$ forms $f$ with components $f_j\in C^{0}(\overline{\Omega})$.
\begin{equation}\label{newdef}
T^{[i_1\dots i_s]}f(z):=\frac{1}{(-2\pi i)^s}\sum_{k=1}^s \sum_{m=0}^{s-1} (-1)^m\sum_{1\leq j_1<\dots <j_m\leq i_{s}}\int_{D_{j_1}\times\dots\times D_{j_m}\times\partial D_{t_1}\times\dots\times\partial D_{t_{s-m-1}}\times D_{i_k}}f_{i_k}\,\frac{\partial^m g_z^{k,i_1\dots i_s}}{\partial\bar{\zeta}_{j_1}\dots\partial\bar{\zeta}_{j_m}}.
\end{equation}
\begin{remark}
When $s=1$, formula \eqref{newdef} recovers the ``slice" operators $T^i$ defined in \eqref{sliceop}. Note that, by construction, if $f$ is $\bar\partial$-closed and with components $f_j\in C^{n-1,\alpha}(\Omega)$, for some $0<\alpha<1$, then \begin{equation}\label{equali}
T^{[i_1\dots i_s]}f=T^{i_1}\dots T^{i_s}\bigg{(}\frac{\partial^{s-1}f_{i_s}}{\partial \bar{z}_{i_1}\dots\partial\bar{z}_{i_{s-1}}}\bigg{)}.
\end{equation}
\end{remark} We are now in a position to extend the operator ${\bf T}$ to all forms $f$ with components $f_j\in C^{0}(\overline{\Omega})$. For such forms, we define 
\begin{equation}\label{newsolutionformula}
{\bf \widetilde{T}}f:=\sum_{s=1}^{n} (-1)^{s-1}\sum_{1\leq i_{1}<\dots<i_s\leq n}T^{[i_1\dots i_s]}f.
\end{equation}
\begin{remark}
In dimension 2, a formula for ${\bf \widetilde{T}}f$ is given by the right side of \eqref{dim2tilde}.
\end{remark}
\begin{remark}\label{ida} By \eqref{equali}, if the form $f$ is $\bar\partial$-closed and has components $f_j\in C^{n-1,\alpha}(\Omega)$ for some $0<\alpha<1$, then 
\begin{equation*}
{\bf \widetilde{T}}f={\bf T}f \quad \text{in }\Omega.
\end{equation*}
\end{remark}

We expect that the operator ${\bf \widetilde{T}}$ will play a role in finding weak solutions to $\bar\partial u=f$ satisfying supnorm estimates on a product domain $\Omega$. Here $f$ is a $(0,1)$ form that is $\bar\partial$-closed in weak sense in $\Omega$ and has components $f_j\in C^0(\overline{\Omega})$. 

\section{Acknowledgements}
The authors acknowledge helpful discussions with Liwei Chen and Jeffery McNeal, and are grateful for their comments on the first draft of this paper. The first author would like to thank the Mathematics Department at Purdue University Fort Wayne for the warm hospitality and financial support.

\end{document}